\documentclass[12pt]{amsart}
\usepackage[margin=1in]{geometry}
\usepackage{graphicx}
\title{Determinants of weighted path matrices}
\author{Kelli Talaska}
\address{University of California, Berkeley}
\email{talaska@math.berkeley.edu}
\thanks{The author was partially supported by NSF Grant DMS-1004532.}

\theoremstyle{plain}
\newtheorem{thm}{Theorem}[section]

\theoremstyle{definition}
\newtheorem{defn}[thm]{Definition}
\newtheorem{example}[thm]{Example}

\numberwithin{equation}{section}

\def\wt{\operatorname{wt}}

\def\sgn{\operatorname{sgn}}
\def\tail{\operatorname{tail}}
\def\P{\mathcal{P}}
\def\PP{\mathbf{P}}
\def\C{\mathcal{C}}
\def\CC{\mathbf{C}}
\def\F{\mathcal{F}}
\def\FF{\mathbf{F}}

\begin{document}

\bibliographystyle{alpha}

\begin{abstract}
We find rational expressions for all minors of the weighted path matrix of a directed graph, generalizing the classical Lindstr\"om/Gessel-Viennot result for acyclic directed graphs. The formulas are given in terms of certain flows in the graph.
\end{abstract}

\maketitle

\section{Introduction}

One of the most elegant results in algebraic combinatorics is that expressing minors of the weighted path matrix of an acyclic graph as a sum over all collections of pairwise vertex-disjoint paths connecting the specified sources and sinks. While this result was (re)discovered and
popularized in combinatorics by Gessel and Viennot \cite{GV1985}, the key ideas can also be found in an earlier, more abstract, formulation by Lindstr\"om \cite{Lindstrom1973} and in Karlin and McGregor's work on coincidence probabilities \cite{KM1959}.

In the acyclic setting, the weighted path matrix has entries which are polynomial in the edge weights of the graph. The main result of this paper is an extension of the above result to graphs which are not necessarily acyclic; in this setting, the weighted path matrix has entries which are formal power series in the edge weights.  In particular, for each minor, we give a rational expression whose numerator and denominator are both polynomials given by simple combinatorial formulas.

The key result in this paper is closely related to work in several other papers. First, we note that Viennot's work on the theory of heaps of pieces \cite{Viennot1986} can be used to derive our rational expressions for the entries of the weighted path matrix, but it does not appear that this approach can be used to obtain the formulas for larger minors.  Fomin's work on loop-erased walks \cite{Fomin2001} also examines the weighted path matrix, giving an expression for each minor as a sum indexed by an infinite but minimal collection of path families satisfying certain intersection criteria.  The path families indexing Fomin's formulas depend on a choice of labeling of the sources and sinks, but our indexing path families are uniquely determined.

The proof of our result involves an involution which generalizes the ``tail-swapping" proof from the acyclic case.  The involution is adapted from the author's work on total positivity in Grassmannians \cite{Talaska2009}, in which planar graphs with directed cycles play a key role, though the analogues of weighted path matrices are rather different in the Grassmannian setting. Besides planarity, there are a number of topological and other technical conditions placed on the graphs used in \cite{Talaska2009}, none of which are necessary in this paper; our formulas hold for all directed graphs.

The classical result for acyclic graphs has many applications, including combinatorial proofs of the Jacobi-Trudi determinantal formulas for Schur functions and MacMahon's formula for the number of plane partitions (see \cite{GV1985} and \cite{Aigner2001} for these and several more well-known examples).  We hope that, with the added flexibility of working with directed cycles, our generalization will have many applications as well. One such example comes from algebraic statistics -- this paper was partly motivated by the author's work with Draisma and Sullivant on Gaussian graphical models. A forthcoming sequel to \cite{STD2010} will use the results of this paper to address models with directed cycles.

\section{Statement of the main theorem}

Let $G$ be a directed graph.  Loops and multiple edges are permitted.  We impose no further conditions on $G$.  In particular, $G$ is permitted to have directed cycles, and we do not assume that $G$ is planar.  Let $V=\{v_1,\ldots,v_n\}$ be the vertex set of $G$.  Assign to each edge $e$ of $G$ the formal variable $x_e$; we call $x_e$ the \emph{weight} of the edge $e$ and assume that all edge weights commute with each other.

A \emph{path} $P=(e_1,e_2,\ldots,e_m)$ in $G$ is formed by traversing the edges $e_1,e_2,\ldots,e_m$ in the specified order.   We write
$P:v\leadsto v'$ to indicate that $P$ is a path starting at a vertex $v$ and ending at a vertex~$v'$.

Define the \emph{weight} of a path $P=(e_1, \ldots, e_m)$ to be
\[
\wt(P)= x_{e_1}\cdots x_{e_m}.
\]

A path $P:v\leadsto v$ with no edges is called a \emph{trivial path} and has weight 1. Paths may or may not have self-intersections; those with no self-intersections will be called \emph{self-avoiding}.

\begin{defn}
\label{def:wt-matrix} The \emph{weighted path matrix} of $G$ is the matrix $M$ whose entries $m_{ij}$ are the formal power series
\begin{equation*}
\label{eq:Aij} m_{ij}=\sum_{P:v_i\leadsto
v_j}\wt(P),
\end{equation*}
the sum over all directed paths $P:v_i \leadsto v_j$.
\end{defn}

This generalizes the classical definition for acyclic graphs.  Note that if $A$ is the weighted adjacency matrix of $G$, then $M=(I-A)^{-1}$.

\begin{figure}[ht]
\begin{center}
\includegraphics{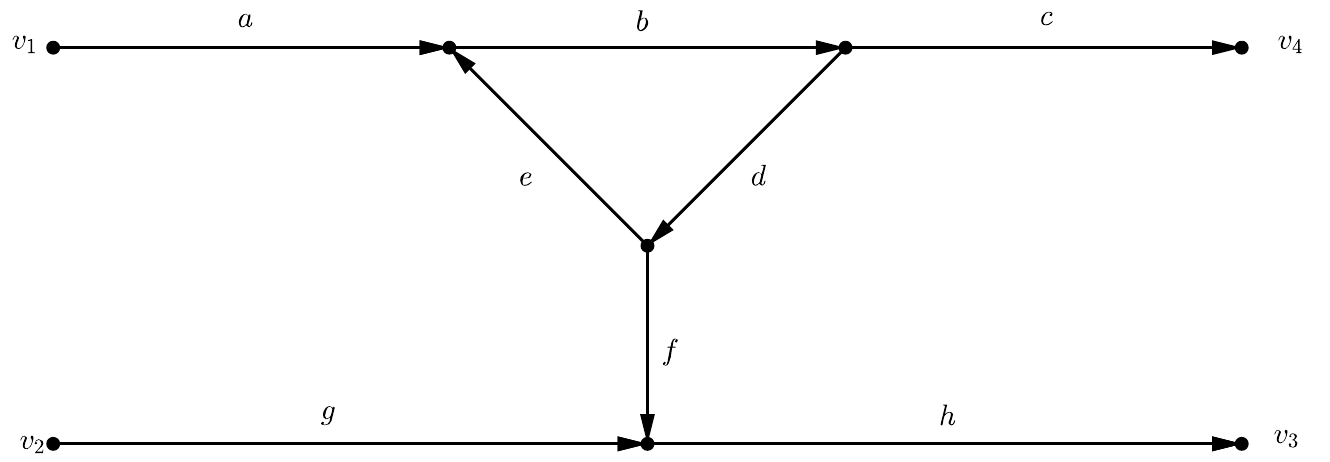}
\end{center}
\caption{An example of a graph $G$ with a directed cycle.}
\label{fig}
\end{figure}

\begin{example}\label{example-main}
Consider the submatrix $M_{\{1,2\},\{3,4\}}$ of the weighted path matrix of the graph $G$ in Figure~\ref{fig}.  For the matrix entry $m_{13}$, we note that any directed path from $v_1$ to $v_3$ must traverse the path with weight $abdfh$, and it may complete an arbitrary number of loops of the cycle of weight $bde$ along the way.  Similarly, for $m_{14}$, any directed path from $v_1$ to $v_4$ must traverse the path with weight $abc$, and it may complete an arbitrary number of loops of the cycle of weight $bde$.  There is a unique path from $v_2$ to $v_3$, with weight $gh$, and there are no paths from $v_2$ to $v_4$.  Thus, we have
\[
M_{\{1,2\},\{3,4\}} =\left(
                      \begin{array}{cc}
                        abdfh(1+bde+(bde)^2+\cdots) & abc(1+bde+(bde)^2+\cdots) \\
                        gh & 0 \\
                      \end{array}
                    \right).
\]
Since this particular example is fairly simple, we can easily see how to write the entries of the matrix as rational expressions:
\[
M_{\{1,2\},\{3,4\}} =\left(
                      \begin{array}{cc}
                        \frac{abdfh}{1-bde} & \frac{abc}{1-bde} \\
                        gh & 0 \\
                      \end{array}
                    \right).
\]
\end{example}

The objective of this paper is to provide such rational expressions for the entries and larger minors of the weighted path matrix when the network may be more complex. We now define several families of paths and cycles needed for the statement and proof of the main theorem.

\begin{defn}\label{def:families}
Let $\C(G)$ denote the set of collections $\CC$ of self-avoiding cycles in $G$ which are pairwise vertex-disjoint.  Each cycle in a nonempty collection $\CC$ much contain at least one edge, and we consider the empty collection an element of $\C(G)$, with weight $1$.

For the remaining families, let $A=\{a_1,\ldots,a_k\}$ and $B=\{b_1,\ldots,b_k\}$ be two subsets of the vertex set $V$, and let $\pi$ be a permutation in the symmetric group $S_k$.

Let $\widetilde{\P}_{A,B,\pi}(G)$ denote the set of path collections $\PP=(P_1,\ldots,P_k)$ such that $P_i$ is a directed path in $G$ from vertex $a_i$ to vertex $b_{\pi(i)}$.  Set $\displaystyle\widetilde{\P}_{A,B}(G)=\bigcup_{\pi\in S_k}\widetilde{\P}_{A,B,\pi}(G)$.

Let $\P_{A,B}(G)$ be the subset of path collections in $\widetilde{\P}_{A,B}(G)$ such that
\begin{itemize}
\item each $P_i$ is self-avoiding, and
\item the paths $P_i$ and $P_j$ are vertex-disjoint whenever $i\neq j$.
\end{itemize}

Let $\F_{A,B}(G)$ denote the set of \emph{self-avoiding flows} connecting $A$ to $B$, i.e., pairs $(\PP,\CC)$ such that
\begin{itemize}
\item $\PP\in\P_{A,B}(G)$,
\item $\CC\in\C(G)$, and
\item $\PP$ and $\CC$ are vertex-disjoint.
\end{itemize}

Informally, a self-avoiding flow in $\F_{A,B}(G)$ is a collection of self-avoiding paths connecting $A$ to $B$ along with a (possibly empty) collection of self-avoiding cycles such that the paths and cycles are all pairwise vertex-disjoint.
\end{defn}

\begin{defn}\label{def:weight}
Define the \emph{weight} of any path collection, cycle collection, or self-avoiding flow to be the product of the weights of the paths and cycles it contains, counting the weights of any repeated edges with the appropriate multiplicities.

Define the \emph{sign} of a path collection, cycle collection, or self-avoiding flow as follows.  If $\PP\in\widetilde{\P}_{A,B,\pi}(G)$, set $\sgn(\PP)=\sgn(\pi)$.  If $\CC\in\C(G)$, let $|\CC|$ denote the number of cycles in $\CC$, and set $\sgn(\CC)=(-1)^{|\CC|}$.  Each self-avoiding flow $\FF$ has a unique decomposition as the disjoint union of a path collection $\PP$ and a cycle collection $\CC$; set $\sgn(\FF)=\sgn(\PP)\sgn(\CC)$.
\end{defn}

\begin{thm}\label{thm:main}
Suppose $G$ is a directed graph with weighted path matrix $M$.  Then the minor $\Delta_{A,B}(M)$, with rows indexed by $A$ and columns indexed by $B$, is given by
\begin{equation}
\label{eq:main-thm}
\Delta_{A,B}(M)=\frac{\displaystyle\sum_{\FF\in\F_{A,B}(G)}\sgn(\FF)\wt(\FF)}{\displaystyle\sum_{\CC\in\C(G)}\sgn(\CC)\wt(\CC)}
\end{equation}
\end{thm}

In general, the fraction in equation (\ref{eq:main-thm}) is not reduced, i.e. the numerator and denominator share a common factor.   Consider the connected components $H_1, H_2, \ldots, H_a$ of the subgraph $G_\circ$ consisting of all edges which appear in at least one directed cycle of $G$.  It can be easily verified that the denominator above can be factored as a product over the connected components of $G_\circ$, i.e.
\[
\sum_{\CC\in\C(G)}\sgn(\CC)\wt(\CC)=\prod_{i=1}^a \sum_{\CC\in \C(H_i)}\sgn(\CC)\wt(\CC).
\]
Each factor $\displaystyle\sum_{\CC\in\C(H_i)}\sgn(\CC)\wt(\CC)$ corresponding to a connected component $H_i$ of $G_\circ$ is irreducible, and it will cancel if and only if every collection of (self-avoiding, pairwise vertex-disjoint) paths in $\P_{A,B}(G)$ avoids the component $H_i$.

\begin{example}
Let us use Theorem~\ref{thm:main} to compute some minors of the weighted path matrix of the graph $G$ in Figure~\ref{fig}. There are two cycle collections in $\C(G)$, the empty collection and the single cycle of weight $bde$, so the denominator of every minor, before canceling common factors, will be $1-bde$.

The matrix entry $m_{13}$ is the $1\times1$ minor $\Delta_{\{1\},\{3\}}(M)$.  There is a unique self-avoiding flow from $v_1$ to $v_3$, namely the path of weight $abdfh$.   Thus, equation~(\ref{eq:main-thm}) tells us that
\[
m_{13}=\frac{abdfh}{1-bde},
\]
which is consistent with our calculations in Example~\ref{example-main}.

For the matrix entry $m_{23}=\Delta_{\{2\},\{3\}}(M)$, we again have a unique self-avoiding path, but in this case there are two self-avoiding flows, since the cycle of weight $bde$ does not intersect the path of weight $gh$.  Thus,
\[
m_{23}=\frac{gh+gh(-bde)}{1-bde}=\frac{gh(1-bde)}{1-bde}=gh,
\]
as we previously showed.

If we look at the $2\times 2$ determinant $|M_{\{1,2\},\{3,4\}}|$, we find a single self-avoiding flow connecting $A=\{a_1=v_1,a_2=v_2\}$ to $B=\{b_1=v_3,b_2=v_4\}$, and the sign of this flow is negative, since the pair of paths connecting $A$ to $B$ sends $a_1$ to $b_2$ and $a_2$ to $b_1$. Equation~(\ref{eq:main-thm}) implies that, as expected, the determinant is
\[
|M_{\{1,2\},\{3,4\}}|=\frac{-abcgh}{1-bde}.
\]
\end{example}

\begin{proof}[Proof of Theorem \ref{thm:main}:]
Using the Leibniz expansion of the determinant, we have
\[
\Delta_{A,B}(M)=\sum_{\pi\in S_k}\sgn(\pi)\prod_{1\leq i\leq k}M_{i,\pi(i)}.
\]
Using Definitions \ref{def:wt-matrix}, \ref{def:families}, and \ref{def:weight}, we can rewrite this as
\[
\Delta_{A,B}(M)=\sum_{\pi\in S_k}\left(\sgn(\pi)\sum_{\PP\in\widetilde{\P}_{A,B,\pi}(G)}\wt(\PP)\right)=\sum_{\PP\in\widetilde{\P}_{A,B}(G)}\sgn(\PP)\wt(\PP).
\]
Thus, to prove Theorem~\ref{thm:main}, it would suffice to show that
\[
\sum_{\CC\in\C(G)}\sum_{\PP\in\widetilde{\P}_{A,B}(G)}\sgn(\PP)\wt(\PP)\sgn(\CC)\wt(\CC)=\sum_{\FF\in\F_{A,B}(G)}\sgn(\FF)\wt(\FF).
\]

That is, all terms on the left cancel except those for which $(\PP,\CC)\in\F_{A,B}(G)$, since each self-avoiding flow $\FF=(\PP,\CC)$ satisfies $\sgn(\FF)\wt(\FF)=\sgn(\PP)\wt(\PP)\sgn(\CC)\wt(\CC)$.  We prove this by constructing a weight-preserving and sign-reversing involution on pairs $(\PP,\CC)\notin \F_{A,B}(G)$.

Suppose that $\PP=(P_1, \ldots, P_k)\in\widetilde{\P}_{A,B}(G)$ and $\CC\in\C(G)$, but $(\PP,\CC)\notin \F_{A,B}(G)$.  This means at least one of the following must be true:
\begin{itemize}
\item there exists an $i$ such that $P_i$ is not self-avoiding,
\item there exists a pair $i\neq i'$ such that $P_i$ and $P_{i'}$ share a common vertex, or
    \item there exists an $i$ such that $P_i$ and $\CC$ share a common vertex.
\end{itemize}

Define $\varphi(\PP,\CC)=(\PP^*,\CC^*)$ as follows. Choose the smallest $i$ such that $P_i$ is not self-avoiding or shares a common vertex with $\CC$ or with some $P_{i'}$ with $i'>i$.  Then, following the algorithm below, we either swap the tails of two paths, as in Figure~\ref{fig:tail-swap}, or we move a cycle from $P_i$ to $\CC$ or vice versa, as in Figure~\ref{fig:move-cycle}.

Let $P_i=(e_1,\ldots, e_m)$.  Let $\tail(e_i)$ denote the tail of the edge $e_i$.  (Each edge points from its tail to its head.)  Choose the smallest $q$ such that the vertex $\tail(e_q)$ lies in $\CC$ or in some $P_{i'}$ with $i'>i$, or $\tail(e_q)=\tail(e_r)$ for some $r>q$.

\begin{itemize}
\item If $\tail(e_q)$ lies in some $P_{i'}$ with $i'>i$, choose the smallest such $i'$.  (This case allows for the possibility that $P_i$ also intersects itself or $\CC$ at the vertex $\tail(e_q)$.)  We then swap the tails of $P_i$ and $P_{i'}$ as follows.  Let $P_{i'}=(h_1, \ldots, h_{m'})$, and choose the smallest $q'$ such that $\tail(h_{q'})=\tail(e_q)$.  Set $P_i^*=(e_1, \ldots,e_{q-1}h_{q'}h_{q'+1}, \ldots,h_{m'})$ and $P_{i'}^*=(h_1, \ldots, h_{q'-1},e_q, e_{q+1},\ldots,e_m)$. Set $\PP^*=\PP\setminus \{P_i,P_{i'}\}\cup\{P_i^*,P_{i'}^*\}$ and set $\CC^*=\CC$.  Note that $\PP^*\neq \PP$, since $P_i$ and $P_{i'}$ have different endpoints.  \footnote{Technically, we have swapped the heads of the two paths, but it is common to refer to such an operation as tail swapping.}

\begin{figure}[ht]
\begin{center}
\includegraphics{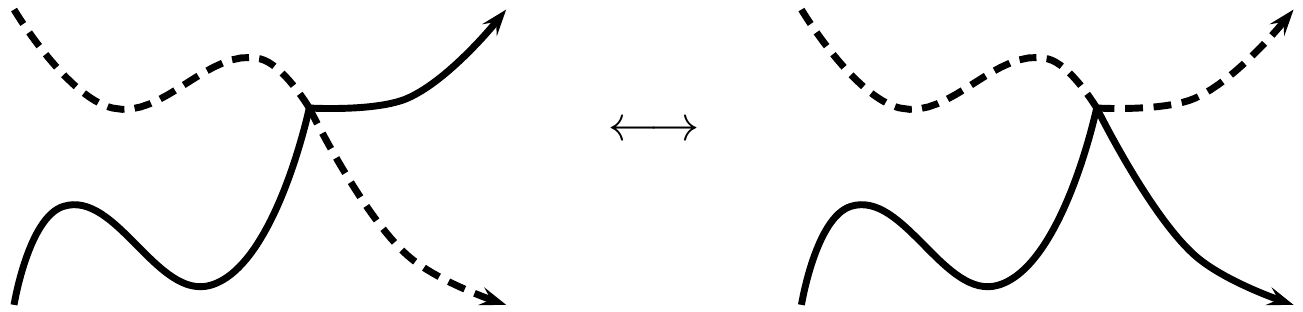}
\end{center}
\caption{Tail swapping.  The paths $P_i$ and $P_i^*$ are drawn as solid curves, and the paths $P_{i'}$ and $P_{i'}^*$ are drawn as dashed curves.}
\label{fig:tail-swap}
\end{figure}

\item Otherwise, we find the first point along $P_i$ where we can move a cycle from $\CC$ to $P_i$ or vice versa, as follows.  If $P_i$ is not self-avoiding, let $\ell$ be the first cycle that $P_i$ completes.  More precisely, choose the smallest $t$ such that $\tail(e_s)=\tail(e_t)$ for some $s<t$; then $\ell=(e_s,e_{s+1},\ldots, e_{t-1})$.  If $P_i$ is self-avoiding, then set $t=\infty$.  If $\CC$ intersects $P_i$, choose the smallest $u$ such that $\tail(e_u)$ appears in a (necessarily unique) cycle $L=(l_1,l_2,\ldots, l_w)$ in $\CC$, where $\tail(l_1)=\tail(e_u)$.  If $\CC$ and $P_i$ are vertex-disjoint, then set $u=\infty$. At least one of $t$ or $u$ must be finite, and $t\neq u$. Note that both $t$ and $u$ may be distinct from $q$, as is the case in Figure~\ref{fig:move-cycle}.
    \begin{itemize}
    \item[$\circ$] If $t<u$, move $\ell$ from $P_i$ to $\CC$.  Set $\CC^*=\CC\cup \{\ell\}$, $P_i^*=(e_1, \ldots, e_{s-1}, e_t, \ldots, e_m)$, and $\PP^*=\PP\setminus \{P_i\}\cup \{P_i^*\}$.
    \item[$\circ$] If $u<t$, move $L$ from $\CC$ to $P_i$, in the following position.  Set $\CC^*=\CC\setminus \{L\}$, $P_i^*=(e_1,\ldots, e_{u-1},l_1, \ldots, l_w, e_u, \ldots, e_m)$, and $\PP^*=\PP\setminus \{P_i\}\cup\{P_i^*\}$.
    \end{itemize}
\begin{figure}[ht]
\begin{center}
\includegraphics{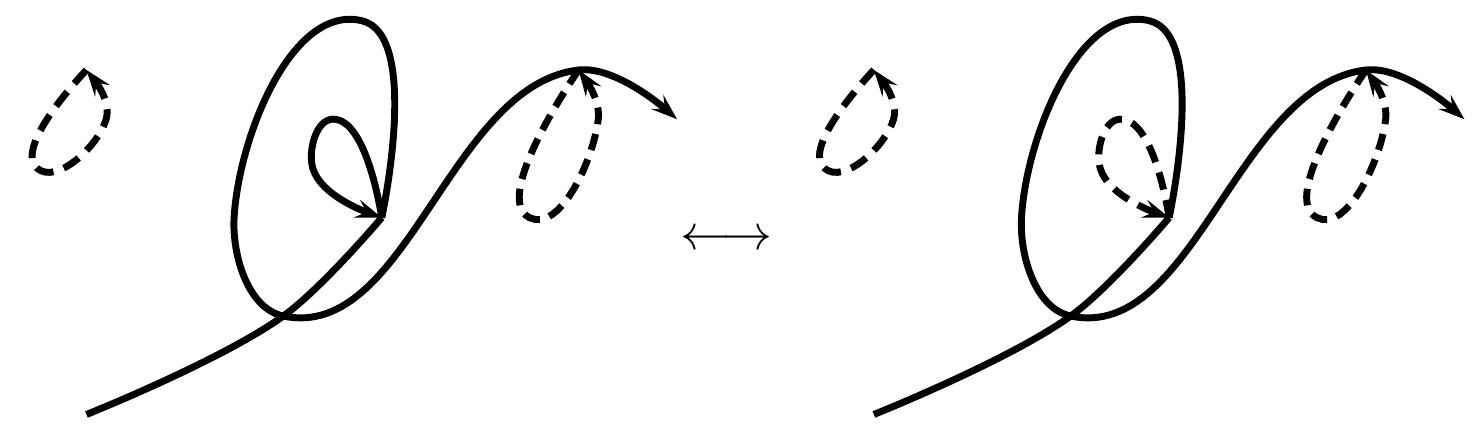}
\end{center}
\caption{Moving a cycle. The paths $P_i$ and $P_i^*$ are drawn as solid curves, and the cycle collections $\CC$ and $\CC^*$ are drawn as dashed cycles.}
\label{fig:move-cycle}
\end{figure}
\end{itemize}

It is easy to see that, with this definition, the image $(\PP^*,\CC^*)$ is again a pair of the required kind, i.e. $\PP^*\in\P_{A,B}(G)$, $\CC^*\in\C(G)$, and $(\PP^*,\CC^*)\notin \F_{A,B}(G)$.

Let us verify that $\varphi$ is an involution. First, we check that $\varphi$ does not change the value of the index $i$.  That is, among all paths in $\PP^*$ which have self-intersections or intersections with other paths in $\PP^*$ or cycles in $\CC^*$, the path with the smallest index is $P_i^*$.  Indeed, our moves only affect $P_i$, $P_{i'}$, and $\CC$, keeping their combined set of edges intact, so the involution will not introduce a new self-intersection in any path $P_j$ such that $j<i$, nor will it introduce a new intersection between such a path $P_j$ and any other path or cycle.

Consider the first case, in which $\varphi(\PP,\CC)=(\PP^*, \CC)$.  After swapping tails, $P_i^*$ still has no intersections with $\CC$ or any of the other paths before the vertex $\tail(e_q)$.  Further, $P_i^*$ does not have any self intersections before $\tail(e_q)$ -- though it may have self-intersections at $\tail(e_q)$ -- since $P_i$ did not have any self-intersections before $\tail(e_q)$ and the tail of $P_{i'}$ did not intersect $P_i$ before $\tail(e_q)$.  Thus $\tail(e_q)$ remains the first vertex along $P_i^*$ with an intersection.  Now, $P_i^*$ and $P_{i'}^*$ intersect at this vertex, and no path with smaller index intersects $P_i^*$ at $\tail(e_q)$, so applying $\varphi$ again swaps the same tails.

Consider the second case, in which $\varphi(\PP,\CC)=(\PP^*,\CC\setminus\{L\})$ or $\varphi(\PP,\CC)=(\PP^*,\CC\cup\{\ell\})$.  Here $P_i$ intersects itself or $\CC$ at the vertex $\tail(e_q)$, but does not intersect any other path at this vertex.  After moving a cycle, the same is true for $P_i^*$. If the cycle moved starts at $\tail(e_q)$, i.e. $\tail(e_q)=\tail(e_t)$ or $\tail(e_q)=\tail(e_u)$, then either a self-intersection becomes an intersection with $\CC$, or an intersection with $\CC$ becomes a self-intersection. If the cycle moved starts later, then the intersections at $\tail(e_q)$ remain the same type after applying $\varphi$. If $P_i$ intersects $\CC$ before completing its first cycle, then $P_i^*$ will complete its first cycle before intersecting $\CC\setminus \{L\}$.  If $P_i$ completes its first cycles $\ell$ before intersecting $\CC$, then $P_i^*$ will intersect $\CC\cup\{\ell\}$ before completing its first cycle.  Thus, the same cycle is moved twice if we apply $\varphi$ twice, which completes the proof that $\varphi$ is an involution.

It remains to verify that $\varphi$ is sign-reversing and weight-preserving.  In the first case, with $\varphi(\PP,\CC)=(\PP^*, \CC)$, the permutations corresponding to $\PP$ and $\PP^*$ differ by a single transposition, so $\sgn(\PP^*)=-\sgn(\PP)$.  The path collections $\PP$ and $\PP^*$ use the same multiset of edges, so $\wt(\PP^*)=\wt(\PP)$.  Since $\CC$ remains fixed, $\sgn(\CC)$ and $\wt(\CC)$ introduce no changes.

In the second case, with $\varphi(\PP,\CC)=(\PP^*, \CC\setminus\{L\})$ or $\varphi(\PP,\CC)=(\PP^*, \CC\cup\{\ell\})$, the endpoints of paths in $\PP$ are the same as those of the corresponding paths in $\PP^*$, so $\sgn(\PP^*)=\sgn(\PP)$.  However $|\CC^*|=|\CC|\pm 1$, so $\sgn(\CC^*)=-\sgn(\CC)$.  Though edges are moved between $\PP$ and $\CC$, the pairs $(\PP, \CC)$ and $(\PP^*,\CC^*)$ both use precisely the same multiset of edges, so $\wt(\PP^*)\wt(\CC^*)=\wt(\PP)\wt(\CC)$.

In each case, the overall sign changes, and the weight is preserved, completing our proof.

\end{proof}

\section*{Acknowledgements}

The author wishes to thank Sergey Fomin, and Lauren Williams for helpful conversations and comments on early versions of this paper.

\bibliography{bib-path-det}

\end{document}